\documentclass[11pt]{amsart}
\usepackage{mathptmx}
\usepackage{amscd,amssymb,amsmath,amsthm,amsfonts}
\usepackage{graphicx}

\newtheorem{theorem}{Theorem}[section]
\newtheorem{prop}[theorem]{Proposition}
\newtheorem{lemma}[theorem]{Lemma}

\newtheorem*{claim*}{Claim}

\theoremstyle{remark}
\newtheorem{remark}{Remark}[section] 
\theoremstyle{definition}
\newtheorem{defin}[remark]{Definition}
\newtheorem*{defin*}{Definition}

\numberwithin{equation}{section}

\DeclareMathAlphabet{\matheur}{U}{eur}{m}{n}
\DeclareMathAlphabet{\matheus}{U}{eus}{m}{n}
\DeclareMathAlphabet{\matheuf}{U}{euf}{m}{n}
\newcommand{\abs}[1]{\left\lvert#1\right\rvert}

\newcommand{\vol}{\mathit{Vol}}
\DeclareMathOperator{\Id}{Id}

\DeclareMathOperator{\dist}{dist}

\DeclareMathOperator{\rank}{rank}
\DeclareMathOperator{\grad}{grad}

\DeclareMathOperator{\diag}{diag}

\usepackage{hyperref}
\author{Ravil Gabdurakhmanov}
\address{School of Mathematics, University of Leeds, Leeds LS2 9JT, UK}
\email{\tt mmrg@leeds.ac.uk}

\author{Gerasim Kokarev}
\address{School of Mathematics, University of Leeds, Leeds LS2 9JT, UK}
\email{\tt G.Kokarev@leeds.ac.uk}



\title{On Calderon's problem for the connection Laplacian}
\date{}

\begin{document}

\begin{abstract}
We consider Calder\'on's problem for the connection Laplacian on a real-analytic vector bundle over a manifold with boundary. We prove a uniqueness result for this problem when all geometric data are real-analytic, recovering the topology and geometry of a vector bundle up to a gauge transformation and an isometry of the base manifold.  
\end{abstract}

\maketitle
\section{Statement and discussion of results}
\subsection{Introduction}
The purpose of this paper is to prove a uniqueness result for the Calder\'on inverse problem for the connection Laplacian on a vector bundle. Our main hypothesis is that the geometry of a vector bundle, that is a connection, a compatible inner product, and a Riemannian metric on the base manifold, are real-analytic. This Calder\'on problem is motivated by the Aharonov-Bohm effect that says that different gauge equivalence classes of electromagnetic potentials have different physical effects that can be detected by experiments. Thus, our uniqueness result shows that different gauge equivalence classes of connections have different boundary data, that is such classes are detectable by boundary measurements.

We discuss some of the related literature on this problem in due course, but now say a few words about the classical Calder\'on problem. Recall that the anisotropic Calder\'on problem asks whether one can read off the conductivity matrix of a medium from electrical voltage and current measurements on the boundary. In dimension greater than two this problem is equivalent to recovering a Riemannian metric on a compact manifold with boundary from the Dirichlet and Neumann data of harmonic functions. A classical result by Lassas and Uhlmann~\cite{LU01}, see also~\cite{LTU03,LLS20}, solves this problem when a Riemannian metric is real-analytic. In more detail, it says that the topology and geometry of a real-analytic Riemannian manifold with boundary can be recovered from the Dirichlet-to-Neumann map for the Laplace-Beltrami operator. The main result of this paper can be viewed as a version of the Lassas-Uhlmann theorem in the setting of vector bundles, which allows us to recover additional topological and geometric structures. We proceed with the statement of related hypotheses and conclusions in more detail.

\subsection{Main result}
Let $(M_i,g_i)$, where $i=1,2$, be two connected compact Riemannian manifolds with boundary, and let $E_i$ be vector bundles over $M_i$. We assume that each $E_i$ is equipped with a connection $\nabla^i$ and a {\em Euclidean structure}, that is a compatible inner product $\langle\cdot,\cdot\rangle_{E_i}$ in the sense that 
$$
D_X\langle u,v\rangle_{E_i}=\langle\nabla^i_Xu,v\rangle_{E_i}+\langle u,\nabla^i_X v\rangle_{E_i}
$$
for any smooth sections $u$ and $v$ of $E_i$, and any vector field $X$ on $M_i$. For open subsets $\Sigma_i\subset\partial M_i$ we denote by $\Lambda_{\Sigma_i}$ the corresponding Dirichlet-to-Neumann maps defined on compactly supported sections of $\left.E_i\right|_{\Sigma_i}$ by the taking the normal derivative of the harmonic extension, that is the solution of the equation
$$
\Delta^{E_i}u=0,\qquad \left.u\right|_{\partial M_i}=s,
$$
where $s$ is a section of $\left.E_i\right|_{\partial M_i}$ supported in $\Sigma_i$, and $\Delta^{E_i}$ is the connection Laplacian on $E_i$, see Section~\ref{prems} for details.

Let $\phi:\left.E_1\right|_{\Sigma_1}\to\left.E_2\right|_{\Sigma_2}$ be a morphism of vector bundles that covers a diffeomorphism $\psi:\Sigma_1\to\Sigma_2$, that is $\pi_2\circ\phi=\psi\circ\pi_1$, where $\pi_i$ is the projection map for $E_i$, $i=1,2$. We say that such a morphism $\phi$ {\em intertwines} with the Dirichlet-to-Neumann maps $\Lambda_{\Sigma_i}$ if 
$$
\phi\circ\Lambda_{\Sigma_1}(s)\circ\psi^{-1}=\Lambda_{\Sigma_2}(\phi\circ s\circ\psi^{-1})
$$
for any smooth section $s$ of $\left.E_1\right|_{\Sigma_1}$. Recall that a vector bundle isomorphism is called {\em Euclidean}, if it preserves Euclidean structures. Our main result is the following theorem.
\begin{theorem}
\label{t1}
Let $(M_i,g_i,E_i,\nabla^i)$, where $i=1,2$, be two Euclidean real-analytic vector bundles defined over connected compact real-analytic Riemannian manifolds with boundary, equipped with real-analytic connections. Suppose that $\dim M_i\geqslant 3$ for each $i=1,2$, and for some open subsets $\Sigma_i\subset\partial M_i$ there exists a real-analytic Euclidean vector bundle isomorphism $\phi:\left.E_1\right|_{\Sigma_1}\to\left.E_2\right|_{\Sigma_2}$ that intertwines with the corresponding Dirichlet-to-Neumann operators $\Lambda_{\Sigma_1}$ and $\Lambda_{\Sigma_2}$. Then the vector bundles $E_1$ and $E_2$ are isomorphic, and moreover, there exists a real-analytic Euclidean vector bundle isomorphism $\Phi:E_1\to E_2$ that covers an isometry $\Psi:(M_1,g_1)\to (M_2,g_2)$, such that $\Phi^*\nabla^2=\nabla^1$ and $\left.\Phi\right|_{\Sigma_1}=\phi$.
\end{theorem}

We note that the presence of Euclidean structures on vector bundles $E_i$ in Theorem~\ref{t1} plays an auxiliary, but important role. In more detail, neither the connection Laplacian nor the associated Dirichlet-to-Neumann operator depend on them, see Section~\ref{prems}. On the other hand, we do not know whether the assumption that the isomorphism $\phi$ in Theorem~\ref{t1} is Euclidean, and the conclusion that so is its extension $\Phi$, can be dropped.

It is an open problem whether the conclusions in Theorems~\ref{t1} hold for arbitrary smooth geometric data, connections and Riemannian metrics, on vector bundles. Under different conditions a similar problem has been considered by Ceki\'c~\cite{Cek17,Cek20}.  Let us also mention that in~\cite{KOP18} the authors consider the Calder\'on problem for the wave operator of the connection Laplacian on Hermitian vector bundles, and obtain conclusions similar to the ones in Theorem~\ref{t1} without any hypotheses on the geometry of vector bundles. 

Note that the hypothesis on the dimension of the base manifolds $M_i$ in our results is essential for the conclusions to hold. In dimension two the connection Laplacian behaves differently when a Riemannian metric on the base changes conformally, see Section~\ref{prems}, and the Riemannian metric on the base can not be recovered. However, as our next theorem shows, a vector bundle with a connection are still determined by the corresponding Dirichlet-to-Neumann operator. To state this result we first need to fix more specific notation.

As is well-known, the conformal class of any smooth metric $g$ on a compact connected surface $M$ determines a real-analytic structure, that is an atlas formed by isothermal charts of $g$. We say that a pair $(M,c)$, where $c$ is a conformal class of Riemannian metrics on $M$, is a {\em real-analytic surface with boundary}, if the real-analytic structure on $M$ is determined by the conformal class $c$, and the boundary is real-analytic with respect to this structure. Let $E_i$, where $i=1,2$, be vector bundles over $M$. Suppose that for an open set $\Sigma\subset\partial M$ there is a morphism of vector bundles $\phi:\left.E_1\right|_{\Sigma}\to\left.E_2\right|_{\Sigma}$  that covers the identity map on the base, that is $\pi_2\circ\phi=\pi_1$, where $\pi_i$ is the projection map for $E_i$, $i=1,2$. Similarly to the notation used above, we say that such a morphism $\phi$ {\em intertwines} with the Dirichlet-to-Neumann maps $\Lambda_{i,\Sigma}$ {\em relative to the conformal class $c$}, if 
$$
\phi\circ\Lambda_{1,\Sigma}(s)=\Lambda_{2,\Sigma}(\phi\circ s)
$$
for any smooth section $s$ of $\left.E_1\right|_{\Sigma}$, where the Dirichlet-to-Neumann operators are defined using some metric $g\in c$. It is straightforward to see that this relation does not depend on a metric $g$ in a fixed conformal class $c$, used to define the Dirichlet-to-Neumann operators, but depends on the conformal class $c$. The following theorem is the version of our main result for vector bundles over surfaces.

\begin{theorem}
\label{t2}
Let $(M,c)$ be a compact connected real-analytic surface with boundary, and let $E_1$ and $E_2$ be two Euclidean real-analytic vector bundles over $M$ equipped with real-analytic connections $\nabla^1$ and $\nabla^2$ respectively. Suppose that for some open subset $\Sigma\subset\partial M$ there exists a real-analytic Euclidean vector bundle isomorphism $\phi:\left.E_1\right|_{\Sigma}\to\left.E_2\right|_{\Sigma}$ that covers the identity map and intertwines with the corresponding Dirichlet-to-Neumann operators $\Lambda_{1,\Sigma}$ and $\Lambda_{2,\Sigma}$ relative to $c$. Then the bundles $E_1$ and $E_2$ are isomorphic, and moreover, there exists a real-analytic Euclidean vector bundle isomorphism $\Phi:E_1\to E_2$ that covers the identity map of $M$, such that $\Phi^*\nabla^2=\nabla^1$ and $\left.\Phi\right|_{\Sigma}=\phi$.
\end{theorem}

Related results on the Calder\'on problem for the connection Laplacian on vector bundles over surfaces can be found in~\cite{AGTU,GT11}, where the authors are concerned with recovering a connection on a fixed vector bundle. Although our hypothesis that the related data are real-analytic might be more restrictive than the ones used in the literature, we do not know any similar uniqueness results for the topology of a vector bundle. It is likely that Theorem~\ref{t2} can be improved, and the topology of the base surface together with the conformal class of metrics on it can also be identified. We plan to address this problem in a future work.

The proofs of Theorems~\ref{t1} and~\ref{t2} build on an elegant idea in~\cite{LTU03}. Using Green kernels for the connection Laplacian, we construct immersions of our vector bundles into some function space, and recover the geometry and topology from their images. We believe that some technical details of our argument in the setting of vector bundles might be of independent interest, and the improvements give a more streamlined proof of the original results in~\cite{LTU03}. A few generalisations of our results are possible. First, Theorem~\ref{t1} can be extended to vector bundles over non-compact complete manifolds with compact boundaries. Second, the conclusions continue to hold for the Dirichlet-to-Neumann operators associated with Schrodinger operators, that is connection Laplacians with symmetric real-analytic potentials, see~\cite{RG22}. It is likely that the hypothesis that potentials are real-analytic can be significantly weakened, but we have made no attempt in pursuing this direction. In~\cite{RG22} the first-named author also obtains analogous results for the Jacobi operator, and applies them to the Calder\'on problem for the harmonic map equation.

\subsection{Organisation of the paper}
The paper is organised in the following way. In Section~\ref{prems} we introduce notation and recall background material on the connection Laplacian and its Green kernel. Here we also recall the results that relate the symbol of the Dirichlet-to-Neumann operator to the geometry on and near the boundary. Section~\ref{immersions} is the main body of the paper. Here we define the immersions of vector bundles via Green kernels to some function space, describe its properties, and prove main results. In Section~\ref{2dim} we outline the version of the main argument for vector bundles over surfaces, which gives the proof of Theorem~\ref{t2}. In the last section we collect proofs of auxiliary statements.

\section{Preliminaries}
\label{prems}
\subsection{Notation}
Let $(M,g)$ be a connected compact Riemannian manifold with boundary, and let $E$ be a vector bundle over $M$. We suppose that $E$ is equipped with a connection $\nabla^E$ and a Euclidean structure, that is an inner product $\langle\cdot,\cdot\rangle_E$ compatible with $\nabla^E$. The latter induces a natural $L_2$-product on the space of smooth sections $\Gamma(E)$ by the formula
$$
(u,v)_2=\int\limits_M\langle u_x,v_x\rangle_E\mathit{dVol}_g(x),
$$
where $\vol_g$ is the volume measure on $M$. Below we use the notation $\Gamma(V)$ for the space of smooth sections of a vector bundle $V$ over $M$. We also have a natural $L_2$-product on $\Gamma(T^*M\otimes E)$, defined by the formula
$$
(\alpha,\beta)_2=\int\limits_M\mathit{trace}_g\langle \alpha_x(\cdot),\beta_x(\cdot)\rangle_E\mathit{dVol}_g(x),
$$ 
where $\mathit{trace}_g$ is the trace (contraction) with respect to a metric tensor $g$ on $M$.

Recall that the {\em connection Laplacian} $\Delta_g^E$ is a second order differential operator on sections of $E$ defined by the formula
\begin{equation}
\label{cl:def}
\Delta^E_g=-\mathit{trace}_g\nabla^2,
\end{equation}
where $\nabla^2:\Gamma(E)\to\Gamma(\otimes^2 T^*M\otimes E)$ is the natural second derivative determined by the Levi-Civita connection on $M$ and the connection $\nabla^E$ on $E$. Equivalently, it can be defined as the composition $(\nabla^E)^*\nabla^E$, where $(\nabla^E)^*$ is the formally adjoint operator with respect to the  $L_2$-products on $\Gamma(E)$ and $\Gamma(T^*M\otimes E)$ defined above, see~\cite[Section~1]{EL83}. Note that although the second definition uses the inner product on $E$, by relation~\eqref{cl:def} the operator $\Delta^E_g$ does not depend on such an inner product. 

Now we briefly discuss the behaviour of the connection Laplacian when a Riemannian metric or a connection change. First, if $\tilde g=\exp(2\varphi) g$ is another Riemannian metric, then a direct computation shows that 
$$
\Delta^E_{\tilde g}s=\exp(-2\varphi)(\Delta^E_{g}s-(n-2)\nabla^E_{Z}s),
$$
where $Z=\grad\varphi$ is the gradient vector field with respect to a metric $g$, and $n$ is the dimension of $M$. In particular, when $n=2$, the operator $\Delta^E_g$ is conformally covariant, and a section $s\in\Gamma(E)$ is harmonic or not with respect to $g$ and $\tilde g$ simultaneously. Second, consider another connection $\tilde\nabla^E$ on $E$, and denote by $\tilde\Delta_g^E$ the corresponding connection Laplacian. Recall that a vector bundle isomorphism $\Phi:E\to E$ is called a {\em gauge equivalence} if $\Phi^*\tilde\nabla^E=\nabla^E$, that is
$$
\nabla^E_Xs=\Phi^{-1}\circ\tilde\nabla_X^E(\Phi\circ s)
$$
for any section $s\in\Gamma(E)$. In a local frame for $E$, this relation is equivalent to
$$
\omega=\gamma^{-1}\tilde\omega\gamma+\gamma^{-1}d\gamma,
$$
where $\omega$ and $\tilde\omega$ are the connection matrices of $\nabla^E$ and $\tilde\nabla^E$ respectively, and $\gamma$ is the matrix of $\Phi$. A straightforward calculation shows that the connection Laplacians of gauge equivalent connections are related by the formula
\begin{equation}
\label{gauge:l}
\Delta_g^E=\Phi^{-1}\circ\tilde\Delta^E_g(\Phi\circ s)
\end{equation}
for any section $s\in\Gamma(E)$. These properties determine the behaviour of other quantities closely related to the  Laplacian, such as its Green kernel and Dirichlet-to-Neumann operator.  We recall the necessary background material on them below.

\subsection{The Green kernel for the connection Laplacian}
Choosing local coordinates on $M$ and a local frame for $E$, it is straightforward to see that the equation 
\begin{equation}
\label{inhom}
\Delta^Eu=w,\quad\text{ where~ }\;w\in\Gamma(E),
\end{equation}
takes the form of an elliptic system of the second order differential equations, whose principal symbol $\sigma(x,\xi)\in\mathcal L(T_xM,T_xM)$  equals $\abs{\xi}^2_g\Id$, where $\xi\in T_x^*M$, $x\in M$. In particular, this system is {\em strongly elliptic in the sense of Petrowsky}, which is the strongest notion of ellipticity for systems, and by the results in~\cite{John59,Lop1,Lop2} there exists a Green matrix locally around every point. For a brief overview of various notions of ellipticity for systems of differential equations and related results we refer to~\cite{Mi70}. In our notation the existence of the Green matrix means that for any point $p\in M$ there exists a neighbourhood $U_p$ and a smooth matrix-valued function $G(x,y)$, where $x,y\in U_p$, $x\ne y$, such that
\begin{equation}
\label{def:rel}
\int\limits_M\langle G(x,y),\Delta^Ew(y)\rangle_{y,E}d\vol_g(y)=w(x)
\end{equation}
for any section $w\in\Gamma(E)$ whose support lies in $U_p$. Above we view $\langle G(x,y),\cdot\rangle_{y,E}$ as a linear operator from the fibre $E_y$ to the fibre $E_x$. For the sequel we need the existence of a similar object globally on $M$.

Below by $E\boxtimes E$ we denote the so-called external tensor product, that is the vector bundle over $M\times M$ whose fibre over a point $(x,y)$ equals $E_x\otimes E_y$. With some abuse of notation, for any vectors $v_x\in E_x$, $u^1_y,u^2_y\in E_y$ we write $\langle v_x\otimes u^1_y, u^2_y\rangle_{y,E}$ for the vector $\langle u^1_y, u^2_y\rangle_Ev_x\in E_x$. Thus, if $G(x,y)$ is viewed as an element in $E_x\otimes E_y$, the notation $\langle G(x,y),\Delta^Ew(y)\rangle_{y,E}$ in relation~\eqref{def:rel} can be understood as an element in $E_x$. We also use the notation $\mathcal D(E)$ for a subspace of $\Gamma(E)$ formed by smooth sections whose supports lie in the interior of $M$.
\begin{defin}
A smooth section $G$ of the vector bundle $E\boxtimes E$ defined away from the diagonal $\diag(M)=\{(x,x)\in M\times M\}$ is called the {\em Dirichlet Green kernel} if: 
\begin{itemize}
\item[(i)] the integral of the function $y\mapsto\abs{G(x,y)}_{E\otimes E}$ is finite for all $x\in M$;
\item[(ii)] relation~\eqref{def:rel} holds for all $w\in\mathcal D(E)$; 
\item[(iii)] $G(x,y)=0$ for all $x\in M$, $y\in\partial M$, $x\ne y$.
\end{itemize}
\end{defin}
It is straightforward to see that the Dirichlet Green kernel occurs as the kernel of a linear operator that sends a section $w\in\Gamma(E)$ to the solution of equation~\eqref{inhom} with the Dirichlet boundary condition $\left.u\right|_{\partial M}=0$. Note that the connection Laplacian $\Delta^E$ on sections of $E$ naturally induces the operator $\Delta^E_y$ on the sections of $E\boxtimes E$ that sends $v_x\otimes u_y$ to $v_x\otimes\Delta^Eu_y$ for any smooth sections $v,u\in\Gamma(E)$. In this notation the condition~(ii) in the definition means that 
$$
\Delta_y^EG(x,\cdot)=\delta_x
$$
for any $x\in M$. For the sequel we note that the combination of the Sobolev embedding together with standard regularity theory shows that the section $y\mapsto G(x,y)$ lies in the Sobolev space $W^{2-k,2}(E)$ for any $k>n/2$, $x\in M$, where $n=\dim M$.

The following statement can be found in~\cite{RG22}; it is a consequence of local existence results in~\cite{John59,Lop1,Lop2}.
\begin{prop}
Let $(M,g,E)$ be a Euclidean vector bundle over a compact Riemannian manifold with boundary, equipped with a compatible connection $\nabla^E$. Then there exists a unique Dirichlet Green kernel on $E$. Besides, it is symmetric in the sense that
$$
G(y,x)=\tau_{x,y}G(x,y)\qquad\text{for all~ }x,y\in M,
$$
where $\tau_{x,y}:E_x\otimes E_y\to E_y\otimes E_x$ is a natural isomorphism that sends $v_x\otimes u_y$ to $u_y\otimes v_x$.
\end{prop}
Now suppose that $(M,g,E)$ is a real-analytic Euclidean vector bundle over a compact real-analytic Riemannian manifold $M$ with boundary, equipped with a real-analytic compatible connection $\nabla^E$. 
Then locally the equation
$$
\Delta_y^EG(x,y)=0,\quad\text{ where }\;\; y\in M\backslash\{x\},
$$
takes the form of a system of differential equations with analytic coefficients, and by~\cite{Pet39,John59} we conclude that the Green kernel is real-analytic in the second variable $y$. More generally, we shall often use the fact that a section of $E$ that is harmonic on an open set is automatically real-analytic on this set.

Let $\tilde\nabla^E$ be another connection on a fixed vector bundle $E$ that is gauge equivalent to $\nabla^E$, that is there exists a vector bundle isomorphism $\Phi:E\to E$ such that $\Phi^*\tilde\nabla^E=\nabla^E$. Denote by $\tilde G$ the Green kernel corresponding to the connection Laplacian $\tilde\Delta^E$. Using relation~\eqref{gauge:l}, it is straightforward to see that $\tilde G=\Phi^\boxtimes\circ G$, where $\Phi^\boxtimes:E\boxtimes E\to E\boxtimes E$ is an isomorphism that equals $\Phi_x\otimes\Phi_y$ on each fibre $E_x\otimes E_y$, where $x$, $y\in M$. In addition, as expected, when dimension of $M$ equals two, the Green kernel is invariant under conformal changes of a metric on $M$.

\subsection{The Dirichlet-to-Neumann operator and its symbol}
As follows from standard theory~\cite{GT}, see also a discussion in~\cite{RG22}, the Dirichlet problem
\begin{equation}
\label{DP}
\Delta^Eu=0,\qquad \left.u\right|_{\partial M}=s,
\end{equation}
has a unique solution for any smooth section $s$ of $\left. E\right|_{\partial M}$. For a given open set $\Sigma\subset\partial M$ the {\em Dirichlet-to-Neumann operator} $\Lambda_\Sigma:\mathcal D(\left. E\right|_{\Sigma})\to \mathcal D(\left. E\right|_{\Sigma})$ is defined on smooth compactly supported in $\Sigma$ sections by the formula
$$
\Lambda_\Sigma(s)=\nabla^E_{\partial/\partial\mathbf{n}}u,
$$ 
where $u$ is the solution to Dirichlet problem~\eqref{DP}, and $\mathbf{n}$ is the outward unit normal vector to $\partial M$. Note that if $\Sigma=\partial M$, and $\Lambda$ is the corresponding Dirichlet-to-Neumann operator, then for any $\Sigma\subset\partial M$ the operator $\Lambda_\Sigma$ is precisely the restriction of $\Lambda$ to the space of sections whose supports lie in $\Sigma$.

Following~\cite{Ta96}, it is straightforward to show, see~\cite{RG22}, that the Dirichlet-to-Neumann operator $\Lambda$ defined above is an elliptic pseudodifferential operator of first order. In more detail, in local coordinated on $\partial M$ that trivialise $\left.E\right|_{\partial M}$ the operator $\Lambda$ takes the form of a pseudodifferential operator whose symbol $p(x,\xi)$, where $x\in M$, $\xi\in T_x^*M$, has an asymptotic expansion
$$
p(x,\xi)\sim \sum_{j\leqslant 1}p_{j}(x,\xi),
$$
where $p_j(x,\xi)$ is a homogeneous matrix-valued function of degree $j$ in $\xi$, and $j\in\mathbb Z$. The asymptotic expansion above is understood in the sense that for any integer $N\geqslant 0$ the matrix-valued function
$p(x,\xi)-\sum_{j=-N}^1p_{j}(x,\xi)$ lies in the class $S^{1-N}_{1,0}$, see~\cite{Shu01,Tre} for details. 

Recall that any local coordinates $(x^1,\ldots,x^{n-1})$ on $\Sigma$ can be extended to the so-called {\em boundary normal coordinates} $(x^1,\dots,x^{n-1},x^n)$ on $M$. The latter are defined by the conditions that the equation $x^n=0$ describes the boundary $\partial M$ and each curve $t\mapsto (x^1,\dots,x^{n-1},t)$ is a unit speed geodesic orthogonal to the boundary. Similarly, any local frame $\{s_l\}$ for $\left. E\right|_{\partial M}$ can be extended to the so-called {\em boundary normal frame} $\{\bar s_l\}$ for $E$ such that
$$
\nabla^E_{\partial/\partial x^n}\bar s_l=0\quad\text{and}\quad \left.\bar s_l\right|_{\Sigma}=s_l,
$$
where $l=1,\ldots,\rank E$.

For the sequel we need the following result that gives a more precise information about the coefficients $p_j(x,\xi)$ and their derivatives. Its proof follows the general strategy used in~\cite{LU89}, and can be found in~\cite{RG21}. In a slightly different notation it also appears in~\cite{Cek20}.
\begin{prop}
\label{loc}
Let $(M,g,E)$ be a vector bundle over a compact Riemannian manifold with boundary, equipped with a connection $\nabla^E$. Suppose that $n=\dim M\geqslant 3$, and let $\Sigma\subset\partial M$ be an open set that trivialises $\left.E\right|_{\partial M}$. For local coordinates $(x^1,\ldots,x^{n-1})$ on $\Sigma$ and a local frame $\{s_\ell\}$ of $E$ over $\Sigma$ let $p_j(x,\xi)$, where $j\leqslant 1$, be a full symbol of the Dirichlet-to-Neumann operator $\Lambda$. Then  the Taylor series of the metric tensor $g$ and the connection matrix $\omega$ of $\nabla^E$ in the boundary normal coordinates $(x^1,\ldots,x^{n-1},x^n)$ and the boundary normal frame $\{\bar s_\ell\}$ at a point $x$ on the boundary are determined by explicit formulae in terms of the matrix-valued functions $p_j(x,\xi)$, where $j\leqslant 1$, and their derivatives at $x$.
\end{prop}

Proposition~\ref{loc} is an important initial ingredient in the proof of our main result, Theorem~\ref{t1}. In particular, it says that the symbol of the Dirichlet-to-Neumann operator determines the metric and connection on the boundary. Moreover, if these data are real-analytic, then the symbol determines them in a neighbourhood of the boundary. It is worth noting that although, in general, the connection Laplacian has a zero order term in a fixed frame, it does not have a natural conformal gauge invariance property, similar to the one for the conformal Laplacian, see~\cite{DKSU09,LLS22}. This can be seen from the explicit formula for the zero term in boundary normal coordinates and boundary normal frame in~\cite{RG21}. In particular, the phenomenon described in~\cite[Theorem~8.4]{DKSU09} does not occur for the connection Laplacian, and the statement of Proposition~\ref{loc} is indeed natural.

In dimension two, the formulae in~\cite{Cek20,RG21} do not allow us to determine the normal derivatives of the metric tensor. However, the Taylor series of the connection matrix can still be recovered from the Dirichlet-to-Neumann operator, assuming that the metric tensor together with all its derivatives is known. More precisely, the following version of Proposition~\ref{loc} holds, see~\cite{RG21}.
\begin{prop}
\label{loc2}
Let $(M,g,E)$ be a vector bundle over a compact Riemannian surface with boundary, equipped with a connection $\nabla^E$, and let $\Sigma\subset\partial M$ be an open set that trivialises $\left.E\right|_{\partial M}$. For a local coordinate $x^1$ on $\Sigma$ and a local frame $\{s_\ell\}$ of $E$ over $\Sigma$ let $p_j(x,\xi)$, where $j\leqslant 1$, be a full symbol of the Dirichlet-to-Neumann operator $\Lambda$. Then  the Taylor series of  the connection matrix $\omega$ of $\nabla^E$ in the boundary normal coordinates $(x^1,x^2)$ and the boundary normal frame $\{\bar s_\ell\}$ at a point $x$ on the boundary are determined by explicit formulae in terms of the matrix-valued functions $p_j(x,\xi)$, where $j\leqslant 1$, the metric tensor $g$, and their derivatives at $x$.
\end{prop}

Now let $(M_i,E_i,\nabla^i)$, where $i=1,2$, be two vector bundles over compact Riemannian manifolds with boundary $(M_i,g_i)$, equipped with connections $\nabla^i$. Suppose that these data are {\em gauge equivalent} in the following sense: there exists a vector bundle isomorphism $\Phi:E_1\to E_2$ that covers an isometry $\Psi:(M_1,g_1)\to (M_2,g_2)$ such that $\Phi^*\nabla^{2}=\nabla^1$. Then, using relation~\eqref{gauge:l} it is straightforward to conclude that the corresponding Dirichlet-to-Neumann operators $\Lambda_1$ and $\Lambda_2$ intertwine, that is
$$
\phi\circ\Lambda_{1}(s)\circ\psi^{-1}=\Lambda_{2}(\phi\circ s\circ\psi^{-1})
$$
for any smooth section $s$ of $\left.E_1\right|_{\partial M_1}$, where $\psi=\left.\Psi\right|_{\partial M_1}$ and $\phi=\left.\Phi\right|_{\partial M_1}$.

Analysing the formulae for the symbols of $\Lambda_1$ and $\Lambda_2$, see~\cite{RG22}, one can show that the converse holds on the boundary.

\begin{prop}
\label{loc:gauge}
Let $(M_i,g_i,E_i,\nabla^i)$, where $i=1,2$, be two Euclidean smooth vector bundles defined over connected compact Riemannian manifolds with boundary. Suppose that for some open subsets $\Sigma_i\subset\partial M_i$ there exists a vector bundle isomorphism $\phi:\left.E_1\right|_{\Sigma_1}\to\left.E_2\right|_{\Sigma_2}$ that intertwines with the corresponding Dirichlet-to-Neumann operators $\Lambda_{\Sigma_1}$ and $\Lambda_{\Sigma_2}$. Then the isomorphism $\phi$ is a gauge equivalence, $\phi^*\nabla^2=\nabla^1$, and covers an isometry $\psi:(\Sigma_1,g_1)\to (\Sigma_2,g_2)$. 
 \end{prop}
The statement above can be viewed as the boundary version of Theorem~\ref{t1}. Note that there is no restriction on dimension of $M$ in Proposition~\ref{loc:gauge}. Similar results continue to hold for the Dirichlet-to-Neumann operator associated with the connection Laplacian with a symmetric real-analytic potential. We refer to~\cite{RG22} for precise statements and a related discussion.

\section{Immersions by Green kernels}
\label{immersions}
\subsection{The construction of immersions}
Let $E$ be a Euclidean real-analytic vector bundle over a connected compact real-analytic manifold $M$ with boundary, equipped with a real-analytic connection $\nabla^E$. In this section we assume that $n=\dim M\geqslant 3$, and describe how one can reconstruct $E$ from the Dirichlet-to-Neumann operator  $\Lambda_\Sigma$, where $\Sigma\subset\partial M$ is an open subset. Our argument develops the ideas from~\cite{LTU03} to the setting of vector bundles, and we aim to make the related technical details to be rather explicit.

Fix a point $p\in\Sigma$. First, note that we may consider $M$ as a subset of a larger real-analytic manifold $\tilde M$. More precisely, choosing boundary normal coordinates $(x^1,\ldots,x^n)$ around $p$, we may identify a neighbourhood of $p$ in $M$ with the Euclidean half-ball
$$
B^+(0,\rho)=\{(x^1,\ldots,x^n)\in B^n(0,\rho): x^n\geqslant 0\},
$$
where $B^n(0,\rho)$ is an open Euclidean ball of radius $\rho>0$ in $\mathbb R^n$. Then, as the manifold $\tilde M$ one can take the manifold obtained by gluing $B^n(0,\rho)$ to $M$ such that points in $B^+(0,\rho)$ are identified with points in $M$ by means of boundary normal coordinates. Below by $U$ we denote the open set $\tilde M\backslash\bar M$. For the sequel it is important to note that the set $U$ does not really depend on $M$. In other words, if there are two manifolds $M_i$ of the same dimension and two points $p_i\in\Sigma_i\subset\partial M_i$, where $i=1,2$, then choosing a sufficiently small $\rho>0$ we may assume that the sets $\tilde M_1\backslash\bar M_1$ and $\tilde M_2\backslash\bar M_2$ coincide.

It is straightforward to see that a real-analytic metric $g$ on $M$ extends to a real-analytic metric $\tilde g$ on $\tilde M$, if $\rho$ is sufficiently small. Similarly, the above construction shows that a real-analytic vector bundle $E$ over $M$ extends to a real-analytic vector bundle $\tilde E$ over $\tilde M$ such that $\left.\tilde E\right|_{U}$ is trivial. Making $\rho>0$ smaller if necessary, we may also assume that a real-analytic Euclidean structure and a real-analytic connection $\nabla$ on $E$ extend to an inner product and a connection $\tilde\nabla$ on $\tilde E$. If the former were compatible on $E$, then by unique continuation so are the latter on $\tilde E$. Below by $\tilde G$ we denote the Dirichlet Green kernel on $\tilde E$.

Denote by $\mathcal E$ the trivial vector bundle $\left.\tilde E\right|_U$. For a given integer $\ell<2-n/2$, where $n$ is the dimension of $M$,  we define the map $\mathcal G:\tilde E\to W^{\ell,2}(\mathcal E)$ by setting
\begin{equation}
\label{g:def}
\tilde E_x\ni v_x\longmapsto \langle v_x,\tilde G(x,\cdot)\rangle_{x,\tilde E}\in W^{\ell,2}(\mathcal E),
\end{equation}
where $x\in \tilde M$. The condition on $\ell$ guarantees that the space $W^{2-\ell,2}_0(\mathcal E)$ embeds
into the H\"older space $C^{0,\alpha}(\mathcal E)$ for some $\alpha>0$, and hence, the dual space $W^{\ell-2,2}(\mathcal E)$ contains the delta function. Then, by elliptic regularity we conclude that $\tilde G(x,\cdot)$ lies in $W^{\ell,2}(\mathcal E)$. In addition, it is straightforward to show that
$$
\abs{\tilde G(x_1,\cdot)-\tilde G(x_2,\cdot)}_{W^{\ell,2}}\leqslant C_1\dist(x_1,x_2)^\alpha
$$
for some constant $C_1>0$, where for simplicity we may assume that the points $x_1$ and $x_2\in\tilde M$ lie in the same chart. Thus, we conclude that the map $\mathcal G$ is continuous. Similarly, we have the following statement.
\begin{lemma}
\label{smooth}
Let $\ell$ be an integer such that $\ell<1-n/2$. Then the map $\mathcal G:\tilde E\to W^{\ell,2}(\mathcal E)$ defined by~\eqref{g:def} is $C^1$-smooth.
\end{lemma}
For the sake of completeness we prove Lemma~\ref{smooth} in Section~\ref{auxs}. It allows us to study the map $\mathcal G$ from a viewpoint of differential geometry. As we shall see below,  the map $\mathcal G$ is a linear embedding on each fibre $\tilde E_x$ for $x\notin \partial\tilde M$, and collapses the fibre to the origin for $x\in\partial\tilde M$. Further, it maps the base manifold $\tilde M$, viewed as the image of the zero section, to zero in $W^{\ell,2}(\mathcal E)$. To avoid these degeneracies we often restrict it to the open set $\tilde E^0$ obtained by considering $\tilde E$ on the interior of $\tilde M$ and removing the zero section. The following statement shows that the map $\mathcal G$ is well-behaved on $\tilde E^0$.
\begin{lemma}
\label{imm}
Let $\ell$ be an integer such that $\ell<1-n/2$. Then the map $\mathcal G:\tilde E\to W^{\ell,2}(\mathcal E)$ defined by~\eqref{g:def} is a linear embedding on each fibre $\tilde E_x$, where $x\notin \partial\tilde M$. Moreover, it is an injective immersion on the set $\tilde E^0$, obtained by removing the image of the zero section from $\tilde E$ over the interior of $\tilde M$.
\end{lemma}

Note that the image of the total space $\tilde E$ under $\mathcal G$ can be viewed as the cone whose link is the image of the subset $S_1\tilde E$ that is formed by vectors of unit length. Then the image of $\mathcal G(\tilde E^0)$ is precisely the set obtained by removing the origin from this cone. By Lemmas~\ref{smooth} and~\ref{imm} it is straightforward to see that the set $\mathcal G(\tilde E^0)$ is a $C^1$-smooth submanifold of $W^{\ell,2}(\mathcal E)$. The main idea behind the proof of Theorem~\ref{t1} is to recover the topology and geometry of $\tilde E$ from this image.

We end this discussion with a lemma that describes another property of the image of $\mathcal G$.
\begin{lemma}
\label{sum}
For given two distinct points $q_1$ and $q_2$ in the interior of $\tilde M$ let $\mathcal W$ be the direct sum $\mathcal G((\tilde E)_{q_1})\oplus\mathcal G((\tilde E)_{q_2})$, viewed as a subspace of $W^{\ell,2}(\mathcal E)$.
Suppose that for some point $x\in\tilde M$ the intersection $\mathcal G((\tilde E)_x)\cap\mathcal W$ is non-trivial. Then the point $x$ has to coincide with one of the points $q_1$ or $q_2$.
\end{lemma}
Proofs of Lemmas~\ref{imm} and~\ref{sum} appear in Section~\ref{auxs}. We continue with a discussion of the main result, Theorem~\ref{t1}.

\subsection{The main result}
Now let $E_i$ be two real-analytic vector bundles over real-analytic manifolds $M_i$, where $i=1,2$, and suppose that for some open sets $\Sigma_i\subset\partial M_i$ there exists a vector bundle isomorphism
$\phi:\left.E_1\right|_{\Sigma_1}\to \left.E_2\right|_{\Sigma_2}$ that intertwines with the Dirichlet-to-Neumann operators $\Lambda_{\Sigma_1}$ and $\Lambda_{\Sigma_2}$. Suppose that $\phi$ covers a diffeomorphism
$\psi:\Sigma_1\to\Sigma_2$, that is $\pi_2\circ\phi=\psi\circ\pi_1$. For a fixed point $p_1\in\Sigma_1$ we set $p_2=\psi(p_1)$, and choose local coordinates on the $\Sigma_i$'s around these points that are related by $\psi$. Note that by Proposition~\ref{loc} the metrics $g_i$ coincide in such coordinates. Thus, making the $\Sigma_i$'s smaller if necessary, we see that the map $\psi:\Sigma_1\to\Sigma_2$ is an isometry.
Since the metrics are real-analytic, by Proposition~\ref{loc} we also conclude that their extensions $\tilde g_i$ coincide in neighbourhoods of the points $p_i$  in $\tilde M_i$. In other words, the isometry $\psi:\Sigma_1\to\Sigma_2$ extends to a real-analytic isometry $\Psi:W_1\to W_2$, defined by identifying boundary normal coordinates, where $W_i$ is a neighbourhood of the point $p_i$ in $\tilde M_i$. In the sequel, we also identify the sets $W_1\backslash \bar M_1$ and $W_2\backslash \bar M_2$, and denote them by $U$.

Similarly, choosing orthonormal frames related by $\phi$, we may identify the trivialisations of $\left.E_1\right|_{\Sigma_1}$ and $\left.E_2\right|_{\Sigma_2}$ around the points $p_1$ and $p_2=\psi(p_1)$. They extend to trivial vector bundles, which we may assume are defined over $W_1$ and $W_2$. Following the discussion above, we obtain vector bundles $\tilde E_1$ and $\tilde E_2$, defined over $\tilde M_1$ and $\tilde M_2$ respectively, and equipped with inner products that are extensions of the original ones on $E_1$ and $E_2$. Making $W_1$ and $W_2$ smaller, if necessary, we may also assume that boundary normal frames of $\tilde E_1$ and $\tilde E_2$ are defined on $W_1$ and $W_2$ respectively. Then, the isomorphism $\phi$ extends to the Euclidean isomorphism 
$$
\Phi:\left.\tilde E_1\right|_{W_1}\to\left.\tilde E_2\right|_{W_2},
$$
defined by identifying the corresponding boundary normal frames. Note that $\Phi$ covers the isometry $\Psi:W_1\to W_2$. By Proposition~\ref{loc} the real-analytic connection matrices of $\nabla^1$ and $\nabla^2$ coincide in such frames, and we conclude that the isomorphism $\Phi$ is a gauge equivalence, that is $\Phi^*\tilde\nabla^2=\tilde\nabla^1$. We continue to use the notation $\mathcal E$ for the vector bundles $\left.\tilde E_i\right|_U$.

Below by $\tilde E^0_i$ we denote the vector bundles $\tilde E_i$ with removed zero sections over the interiors of $\tilde M_i$ , where $i=1,2$. By the hypotheses of Theorem~\ref{t1} the dimensions of the base manifolds $M_i$ coincide; we denote this integer by $n$. Theorem~\ref{t1} is a consequence of the following statement. 
\begin{theorem}
\label{t1a}
Under the hypotheses of Theorem~\ref{t1}, consider the maps $\mathcal {G}_i:\tilde E_i\to W^{\ell,2}(\mathcal E)$ defined by~\eqref{g:def}, where $i=1,2$, and $\ell$ is an integer such that $\ell<1-n/2$. Suppose that the vector bundle isomorphism $\Phi:\mathcal E\to\mathcal E$, described above, intertwines with the $\mathcal G_i$'s, that is 
\begin{equation}
\label{intert}
\mathcal G_2\circ\Phi=\Phi\circ\mathcal G_1\qquad\text{ on~ }~\mathcal E. 
\end{equation}
Then the images $\mathcal G_2(\tilde E_2^0)$ and $\Phi\circ\mathcal G_1(\tilde E_1^0)$ coincide as subsets in $W^{\ell,2}(\mathcal E)$, and the map $\mathcal G_2^{-1}\circ\Phi\circ\mathcal G_1:\tilde E_1^0\to \tilde E_2^0$ extends  to a real-analytic vector bundle isomorphism $J:\tilde E_1\to\tilde E_2$ that covers an isometry $j:\tilde M_1\to\tilde M_2$ such that $J^*\tilde\nabla^2=\tilde\nabla^1$.
\end{theorem}

Now we show how Theorem~\ref{t1a} implies Theorems~\ref{t1}.

\begin{proof}[Proof of Theorem~\ref{t1}]
First, since the vector bundle isomorphism $\phi:\left.E_1\right|_{\Sigma_1}\to\left.E_2\right|_{\Sigma_2}$ preserves inner products on $E_1$ and $E_2$, then so does its extension $\Phi:\mathcal E\to\mathcal E$. This statement follows directly from the definition of $\Phi$ as an isomorphism that identifies boundary normal frames. Now we claim that the conclusion of Theorem~\ref{t1a} implies Theorem~\ref{t1}. Indeed, by relation~\eqref{intert}, we see that the vector bundle isomorphism $J:\tilde E_1\to\tilde E_2$ coincides with $\Phi$ on the set $\left.\tilde E_1\right|U$, and the isometry $j:\tilde M_1\to\tilde M_2$ coincides with $\Psi$ on $U$. Thus, they are genuine extensions of the isomorphism $\phi$ and the isometry $\psi$ from the boundary, and satisfy the conclusions of Theorem~\ref{t1}. Since $\Phi$ preserves the inner products, we conclude that the products $\langle\cdot,\cdot\rangle_{\tilde E_1}$ and $J^*\langle\cdot,\cdot\rangle_{\tilde E_2}$ coincide on $\left.\tilde E_1\right|U$, and hence, by unique continuation coincide everywhere on $\tilde E_1$. Thus, the isomorphism $J$ preserves inner products, and its restriction to $E_1$ satisfies all conclusions of Theorem~\ref{t1}.

For a proof of Theorem~\ref{t1} we need to prove relation~\eqref{intert}, that is the vector bundle isomorhism $\Phi:\mathcal E\to\mathcal E$ intertwines with the immersions $\mathcal G_i$'s. Since $\Phi$ preserves Euclidean structures, for the latter it is sufficient to show that
$$
\tilde G_2(\Psi(x),\Psi(y))=\Phi^\boxtimes\tilde G_1(x,y)\qquad\text{ for all~ } (x,y)\in U\times U.
$$
Choosing coordinates on $W_1$ and $W_2$ related by $\Psi$, we may assume that $\Psi:W_1\to W_2$ is the identity. Similarly, choosing local trivialisations of the $\left.\tilde E_i\right|_{W_i}\simeq\mathcal E$ related by $\Phi$, we assume that so is $\Phi$. Thus, it remains to show that the Green matrices $\tilde G_1$ and $\tilde G_2$, viewed as sections of the trivial bundle $\mathcal E\boxtimes\mathcal E$, coincide. For classical Green functions, that is when the rank of $\mathcal E$ equals one, this statement is well known, see~\cite[Lemma~2.1]{LTU03}. It is a consequence of standard regularity theory together with uniqueness of Dirichlet Green functions. Below we outline a version of this argument in our setting.

First, since under our assumptions  the isomorphism $\Phi$ is the identity on $\mathcal E$, the hypothesis in Theorem~\ref{t1} means that the Dirichlet-to-Neumann operators $\Lambda_1$ and $\Lambda_2$ restricted to sections supported in $W_1\cap\partial M_1$ and $W_2\cap\partial M_2$ respectively, coincide. Pick a point $x\in U$, and for a non-zero vector  $v_x$ in the fibre $\mathcal E_x$ consider  a solution $s$ to the Dirichlet problem
$$
\Delta^{E_2}s=0,\qquad \left.s\right|_{\partial M_2}=\langle v_x,\tilde G_1(x,\cdot)\rangle_{x,\tilde E_1},
$$
on $M_2$. We define a continuous section $\tilde s$ of $\tilde E_2$ away from $x$ by extending $s$ as $\langle v_x,\tilde G_1(x,\cdot)\rangle_{x,\tilde E_1}$ on $U\backslash\{x\}$. Note that the section $\langle v_x,\tilde G_1(x,\cdot)\rangle_{x,\tilde E_1}$ solves the Dirichlet problem
$$
\Delta^{E_1}s=0,\qquad \left.s\right|_{\partial M_1}=\langle v_x,\tilde G_1(x,\cdot)\rangle_{x,\tilde E_1},
$$
and since the Dirichlet-to-Neumann operators coincide, we conclude that so do the normal derivatives of $s$ and $\langle v_x,\tilde G_1(x,\cdot)\rangle_{x,\tilde E_1}$ on the boundary $W_2\cap\partial M_2$. Thus, the section $\tilde s$ is $C^1$-smooth, and the standard application of Green's formulae shows that $\tilde s$ is weakly harmonic on $\tilde M_2\backslash\{x\}$, and hence, is smooth. Since a vector $v_x\in\mathcal E_x$ is arbitrary, and the Euclidean structures agree, this construction yields a smooth section $H(x,y)\in (\tilde E_2)_x\otimes (\tilde E_2)_y$ such that:
\begin{itemize}
\item $\Delta^{E_2}_yH(x,y)=0$ for $y\in M_2$;
\item $H(x,y)=\tilde G_1(x,y)$ for $y\in U$, $y\ne x$;
\item $H(x,y)=0$ for $y\in\partial\tilde M_2$.
\end{itemize}
In particular, we see that $\Delta^{E_2}_yH(x,\cdot)=\delta_x$ on $\tilde M_2$, and the standard argument used to prove uniqueness of the Dirichlet Green kernel shows that $H(x,y)$ coincides with the Dirichlet Green kernel $\tilde G_2(x,y)$ for all $y\in\tilde M_2$. Thus, the Green matrices $\tilde G_1(x,\cdot)$ and $\tilde G_2(x,\cdot)$ indeed coincide on the set $U\backslash\{x\}$, and we are done.
\end{proof}

\subsection{Proof of Theorem~\ref{t1a}}
We start with outlining the general strategy of a proof. Let $B_1\subset\tilde M_1$ be the largest connected open set containing the fixed point $p_1\in\Sigma_1$ and such that for any $x\in B_1$ there exists a unique $j(x)\in \tilde M_2$ such that the images of fibres $\Phi\circ\mathcal G_1((\tilde E_1)_x)$ and $\mathcal G_2((\tilde E_2)_{j(x)})$ coincide and the operator
\begin{equation}
\label{J:def}
J_x=\mathcal G_2^{-1}\circ\Phi\circ\mathcal G_1: (\tilde E_1)_x\longrightarrow (\tilde E_2)_{j(x)}
\end{equation}
is an isometry with respect to the inner products on the fibres. Note that if the subspaces $\Phi\circ\mathcal G_1((\tilde E_1)_x)$ and $\mathcal G_2((\tilde E_2)_{j(x)})$ coincide, by Lemma~\ref{imm} the map $J_x$, defined in~\eqref{J:def}, is automatically an isomorphism of the fibres, and defines a fibre preserving map $J:\left.\tilde E_1\right|_{B_1}\to\tilde E_2$. 

First, we claim that the set $B_1$ contains the neighbourhood $W_1$ of $p_1$ constructed above. Indeed, since $\Phi$ intertwines with the maps $\mathcal G_1$ and $\mathcal G_2$ on $\mathcal E$, we have
\begin{equation}
\label{intert:ext}
\Phi_y\langle v_x,\tilde G_1(x,y)\rangle_{x}=\langle\Phi_{\Psi(x)}v_{\Psi(x)},\tilde G_2(\Psi(x),\Psi(y))\rangle_{x}
\end{equation}
for all $x,y\in U$, and $v_x\in(\tilde E_1)_x$. Choosing a real-analytic non-zero section $v$ on $W_1$, since both sides in the relation above are real-analytic, we conclude that this relation continues to hold for all $x\in W_1$, $y\in U$. Since $v_x$ may take arbitrary values, it is straightforward to see that for any $x\in W_1$ we may choose $\Psi(x)$ as the point $j(x)$ in the definition of the set $B_1$. Indeed, relation~\eqref{intert:ext} implies that for any $x\in W_1$ the operator $J_x$ coincides with $\Phi:(\tilde E_1)_x\to(\tilde E_2)_{\Psi(x)}$, which is an isometry by its own definition. By Lemma~\ref{imm} it is a unique point that satisfies this condition. Thus, the set $W_1$ indeed lies in $B_1$.

Our main aim is to show that the set $B_1$ coincides with $\tilde M_1$. Once this statement is proved, we shall show that the map $J:\tilde E_1\to\tilde E_2$, defined on each fibre by relation~\eqref{J:def}, is a vector bundle isomorphism that satisfies the conclusions of the theorem.

Suppose the contrary, $B_1\ne\tilde M_1$. Then there exists a point $x_1\in\partial B_1$ that lies in the interior of $\tilde M_1$, that is $x_1\notin\partial\tilde M_1$. Since $W_1\subset B_1$, the point $x_1$ lies in the complement $\tilde M_1\backslash W_1$, and in particular, we see that $x_1\notin\bar U$. 

\noindent\underline{\em Step~1.}
First, we claim that the map $J$ can be extended to the fibre $(\tilde E_1)_{x_1}$ over $x_1$.
\begin{lemma}
\label{step1}
Let $x_1\in\partial B_1$ be a point such that $x_1\notin\partial\tilde M_1$. Then there exists a unique point $x_2$ in the interior of $\tilde M_2$ such that the images of fibres $\Phi\circ\mathcal G_1((\tilde E_1)_{x_1})$ and $\mathcal G_2((\tilde E_2)_{x_2})$ coincide, and the corresponding operator $J_{x_1}$ is an isometry. Moreover, for any non-zero vector $v_{x_1}\in(\tilde E_1)_{x_1}$ there exists a unique non-zero vector $w_{x_2}\in(\tilde E_2)_{x_2}$ such that
$$
\Phi\circ\mathcal G_1(v_{x_1})=\mathcal G_2(w_{x_2}),\qquad \abs{v_{x_1}}_{\tilde E_1}=\abs{w_{x_2}}_{\tilde E_2},
$$
and for any converging sequence $v_{p_k}\to v_{x_1}$, where $p_k\to x_1$, we have $J(v_{p_k})\to w_{x_2}$ as $k\to +\infty$.
\end{lemma}
\begin{proof}
Let $p_k\in B_1$ be a sequence of points that converges to the point $x_1\in\partial B_1$, and $q_k$ the corresponding sequence of points such that the images of fibres $\Phi\circ\mathcal G_1((\tilde E_1)_{p_k})$ and $\mathcal G_2((\tilde E_2)_{q_k})$ coincide. Since $\tilde M_2$ is compact, then choosing a subsequence, which we denote by the same symbol $q_k$, we may assume that $q_k\to q_0\in\tilde M_2$ as $k\to+\infty$. For a non-zero vector $v_{x_1}\in(\tilde E_1)_{x_1}$ pick a sequence $v_k\in(\tilde E_1)_{p_k}$ that converges to $v_{x_1}$, and let $w_k\in (\tilde E_2)_{q_k}$ be the corresponding sequence such that
$$
\Phi\circ\mathcal G_1(v_k)=\mathcal G_2(w_k)\quad\text{ and }\quad \abs{v_k}_{\tilde E_1}=\abs{w_k}_{\tilde E_2}.
$$
Since the sequence $w_k$ is bounded, we may assume, again after choosing a subsequence, that $w_k$ converges to some vector $w_{q_0}\in (\tilde E_2)_{q_0}$ as $k\to+\infty$. It is straightforward to see that the norm of ${w_{q_0}}$ equals the one of ${v_{x_1}}$. Now for a proof of the lemma it remains to show that $q_0\notin\partial\tilde M_2$. If the latter holds, then we may take $q_0$ as $x_2$, and the statement follows directly by continuity of $\Phi\circ\mathcal G_1$ and $\mathcal G_2$. The uniqueness of the point $x_2$ and the vector $w_{x_2}$ is a consequence of Lemma~\ref{imm}.

Suppose the contrary, $q_0\in\partial\tilde M_2$. Then by continuity we obtain
$$
\Phi\circ\mathcal G_1(v_{x_1})=\lim\Phi\circ\mathcal G_1(v_k)=\lim\mathcal G_2(w_k)=\mathcal G_2(w_{q_0}).
$$
Since the point $x_1$ lies in the interior of $\tilde M_1$, by Lemma~\ref{imm} the left-hand side above is non-zero, while since $q_0\in\partial\tilde M_2$, the right-hand side vanishes. Thus, we arrive at a contradiction.
\end{proof}

\noindent\underline{\em Step~2.}
Now we analyse the images $\mathcal R_i$ of the maps $\mathcal G_i$ in $W^{\ell,2}(\mathcal E)$, where $i=1,2$. Take a non-zero vector $v_{x_1}\in(\tilde E_1)_{x_1}$, and let $x_2\in\tilde M_2$ and $w_{x_2}\in(\tilde E_2)_{x_2}$ be a point and a vector respectively that satisfy the conclusions of Lemma~\ref{step1}. In particular, the vectors $\Phi\circ\mathcal G_1(v_{x_1})$ and $\mathcal G_2(w_{x_2})$ coincide in $W^{\ell,2}(\mathcal E)$, and we denote this value by $u$. By Lemma~\ref{imm} we see that locally the sets $\Phi(\mathcal R_1)$ and $\mathcal R_2$ are submanifolds in $W^{\ell,2}(\mathcal E)$, whose tangent spaces can be viewed as the images of the differentials $D(\Phi\circ\mathcal G_1)$ and $D\mathcal G_2$. Combining this with Lemma~\ref{step1}, we conclude that the tangent spaces $T_u\Phi(\mathcal R_1)$ and $T_u\mathcal R_2$ coincide as subspaces in $W^{\ell,2}(\mathcal E)$. Using the inverse function theorem we may view $\Phi(\mathcal R_1)$ and $\mathcal R_2$ locally near $u$ as graphs of smooth functions defined on an open subset in
$$
\mathcal V=T_u\Phi(\mathcal R_1)=T_u\mathcal R_2.
$$
In more detail, let $\Pi$ be the orthogonal projection onto $\mathcal V$ in $W^{\ell,2}(\mathcal E)$, and consider the map 
$$
\Pi\circ\mathcal G_2:\tilde E_2\to\mathcal V,\qquad v_x\longmapsto \Pi(\langle v_x,\tilde G_2(x,\cdot)\rangle).
$$
By Lemma~\ref{imm}, its differential is an isomorphism near $u$, and hence, there exists a $C^1$-smooth inverse map $H_2:\mathcal O\to\tilde E_2$, defined in the neighbourhood $\mathcal O$ of $\Pi (u)$ in $\mathcal V$. Then, it is straightforward to see that near $u$ the image $\mathcal R_2$ is the graph of the map
$$
F_2:\mathcal O\to\mathcal V^\bot,\qquad \upsilon\longmapsto\mathcal G_2(H_2(\upsilon))-\upsilon,
$$
where $\mathcal V^\bot$ is the orthogonal complement of $\mathcal V$ in $W^{\ell,2}(\mathcal E)$. Similarly, one shows that there exists a $C^1$-smooth map $H_1:\mathcal O\to\tilde E_1$, which we may assume is defined on the same set $\mathcal O$, such that near $u$ the image $\Phi(\mathcal R_1)$ is the graph of the map
$$
F_1:\mathcal O\to\mathcal V^\bot,\qquad \upsilon\longmapsto\Phi\circ \mathcal G_1(H_1(\upsilon))-\upsilon.
$$
From this construction we see that the vectors $v_{x_1}\in(\tilde E_1)_{x_1}$ and $w_{x_2}\in(\tilde E_2)_{x_2}$ are precisely the images  $H_1\circ\Pi(u)$ and $H_2\circ\Pi(u)$, and moreover, the isomorphism $J$ has the form $H_2\circ H_1^{-1}$ on the open subset
\begin{equation}
\label{omega1}
\Omega_1=H_1(\mathcal O)\cap\tilde\pi_1^{-1}(B_1)\subset\tilde E_1,
\end{equation}
where $\tilde\pi_1:\tilde E_1\to\tilde M_1$ is the vector bundle projection.

For the sequel we need the following lemma.
\begin{lemma}
\label{step2}
The maps $H_i:\mathcal O\to\tilde E_i$ constructed above, where $i=1,2$, are real-analytic in a neighbourhood of $\Pi(u)$ in $\mathcal V$. In particular, there exists a neighbourhood of $v_{x_1}$ in $\tilde E_1$ such that the map $H_2\circ H_1^{-1}$ is real-analytic on it.
\end{lemma}
\begin{proof}
Choosing an orthonormal basis $(\varphi_i)$ in $\mathcal V$, where $i=1,\ldots, m$, we may identify the vector space $\mathcal V$ with $\mathbb R^m$. First, we claim that the map $\Pi\circ\mathcal G_2:\tilde E_2\to\mathcal V\simeq\mathbb R^m$ is real-analytic in a neighbourhood of $w_{x_2}$, that is the coordinate functions, given by products
$$
(\Pi\circ\mathcal G_2,\varphi_i)_{\ell,2}=(\mathcal G_2,\varphi_i)_{\ell,2},\qquad\text{where}\quad i=1,\ldots, m,
$$ 
and $(\cdot,\cdot)_{\ell,2}$ stands for the scalar product in $W^{\ell,2}(\mathcal E)$, are real-analytic. By definition of $\mathcal G_2$ for the latter it is sufficient to show that the sections
$$
x\longmapsto (\tilde G_2(x,\cdot),\varphi_i)_{\ell,2}\in E_x\qquad\text{where}\quad i=1,\ldots, m,
$$
are real-analytic in a neighbourhood of $x_2$. Let $f_i\in W^{-\ell,2}_0(\mathcal E)$ be a vector dual to $\varphi_i$, that is such that $\varphi_i(s)=(s,f_i)_{-\ell,2}$ for any $s\in W^{-\ell,2}_0(\mathcal E)$. Since the canonical map $f\mapsto (\cdot,f)_{-\ell,2}$ preserves scalar products, we conclude that
$$
(\tilde G_2(x,\cdot),\varphi_i)_{\ell,2}=\int\limits_U\langle\tilde G_2(x,y),f_i(y)\rangle_{y,\tilde E_2}\mathit{dVol}_g(y).
$$
Recall that the point $x_1$ does not lie in the closure $\bar U\subset\tilde M_2$. Then, by properties of the Green kernel, it is straightforward to see that the integral on the right-hand side above defines a harmonic section in any neighbourhood of $x_1$ that is disjoint with $U$. As was discussed in Section~\ref{prems}, any harmonic section is real-analytic under our hypotheses, and we conclude that so is the integral above. Thus, the coordinate functions $(\mathcal G_2,\varphi_i)_{\ell,2}$ are real-analytic in a neighbourhood of $x_1$ for all $i=1,\ldots, m$. Further, we conclude that the map $H_2$, as the inverse map to $\Pi\circ\mathcal G_2$, is also real-analytic in a neighbourhood of $\Pi(u)$.

A similar argument shows that the maps $\Pi\circ\Phi\circ\mathcal G_1$ and $H_1$ are real-analytic as well. Hence, the map $H_2\circ H_1^{-1}$ is real-analytic as the composition of real-analytic maps.
\end{proof}

\noindent\underline{\em Step~3.} Now we claim that the images of $\Phi(\mathcal R_1)$ and $\mathcal R_2$ coincide around the point $u$. This is the consequence of the following lemma.
\begin{lemma}
\label{step3a}
The maps $F_i:\mathcal O\to\mathcal V^\bot$ constructed above, where $i=1,2$, coincide in a neighbourhood of $\Pi(u)$ in $\mathcal V$.
\end{lemma}
\begin{proof}
Fix an orthonormal basis $(\varphi_j)$ in $\mathcal V^\bot$, where $j=1,2,\ldots,\infty$. For a proof of the lemma it is sufficient to show that the coordinate functions $(F_1,\varphi_j)_{\ell,2}$ and $(F_2,\varphi_j)_{\ell,2}$ coincide for all $j=1,2,\ldots,\infty$, where $(\cdot,\cdot)_{\ell,2}$ is the scalar product in $W^{\ell,2}(\mathcal E)$. Note that
\begin{equation}
\label{eq:f2}
(F_2,\varphi_j)_{\ell,2}(\upsilon)=(\mathcal G_2,\varphi_j)_{\ell,2}\circ H_2(\upsilon)-(\upsilon,\varphi_j)_{\ell,2}
\end{equation}
for any $\upsilon\in\mathcal O$. The argument used in the proof of Lemma~\ref{step2} shows that the function $(\mathcal G_2,\varphi_j)_{\ell,2}$ is real-analytic in some neighbourhood of $v_{x_1}$, and by Lemma~\ref{step2} we also know that the map $H_2$ is real-analytic in a neighbourhood of $\Pi(u)$. Since the second term on the right-hand side of~\eqref{eq:f2} is linear in $\upsilon$, we conclude that the function $(F_2,\varphi_j)_{\ell,2}$ is real-analytic in a neighbourhood of $\Pi(u)$, which we may also denote by $\mathcal O$. This statement holds for all values $j=1,2,\ldots,\infty$, with the same set $\mathcal O$. 

Similarly, one shows that all functions 
\begin{equation}
\label{eq:f1}
(F_1,\varphi_j)_{\ell,2}(\upsilon)=(\Phi\circ\mathcal G_1,\varphi_j)_{\ell,2}\circ H_1(\upsilon)-(\upsilon,\varphi_j)_{\ell,2}
\end{equation} 
are also real-analytic on the same set $\mathcal O$. Without loss of generality, we may assume that the open set $\mathcal O$ is connected. Now by the choice of the point $x_1$, we know that the maps $\Phi\circ\mathcal G_1$ and $\mathcal G_2\circ J$ coincide on an open subset $\Omega_1\subset\tilde E_1$, defined in~\eqref{omega1}, whose closure contains $v_{x_1}$. Recall that the map $J$ coincides with $H_2\circ H_1^{-1}$ on $\Omega_1$, and hence, the maps $\Phi\circ\mathcal G_1\circ H_1$ and $\mathcal G_2\circ H_2$ coincide on $H_1^{-1}(\Omega_1)\subset\mathcal O$. Combining the latter with relations~\eqref{eq:f2} and~\eqref{eq:f1}, we conclude that the real-analytic functions $(F_1,\varphi_j)_{\ell,2}$ and $(F_2,\varphi_j)_{\ell,2}$ coincide on an open subset $H_1^{-1}(\Omega_1)\subset\mathcal O$, and hence, by unique continuation coincide on $\mathcal O$ for all $j=1,2,\ldots,\infty$. Thus, we are done.
\end{proof}

Due to conical structure of the images $\Phi(\mathcal R_1)$ and $\mathcal R_2$, from the above we conclude that there are conical neighbourhoods of $u$, that is neighbourhoods invariant under multiplication by $t>0$, that coincide. In fact, as the following lemma shows, even a stronger statement holds.
\begin{lemma}
\label{step3b}
There is a neighbourhood $O_1$ of the point $x_1\in\tilde M_1$ such that for any $x\in O_1$ there exists $z\in\tilde M_2$ such that the images of fibres $\Phi\circ\mathcal G_1((\tilde E_1)_x)$ and $\mathcal G_2((\tilde E_2)_{z})$ coincide.
\end{lemma}
\begin{proof}
Choose a neighbourhood $O_1$ of $x_1\in\tilde M_1$ such that $O_1\subset\tilde\pi_1\circ H_1(\mathcal O)$, where $\tilde\pi_1:\tilde E_1\to\tilde M_1$ is the vector bundle projection. We intend to show that for any $x\in O_1$ there exists $z\in\tilde M_2$ such that the image $\Phi\circ\mathcal G_1((\tilde E_1)_x)$ lies in $\mathcal G_2((\tilde E_2)_{z})$. Since these images are vector spaces of the same dimension, the statement of the lemma follows immediately.

First, for a given point $x\in O_1$ and a vector $v_x\in(\tilde E_1)_x$, the considerations above show that the image $\Phi\circ\mathcal G_1(v_x)$ lies in the set $\mathcal G_2(CH_2(\mathcal O))$, where $CH_2(\mathcal O)$ is a conical open set,
$$
CH_2(\mathcal O)=\{tw\in\tilde E_2:~ t\in\mathbb R,~ t>0, \text{ and } w\in H_2(\mathcal O) \}.
$$
Thus, there exists $z\in\tilde M_2$ such that $\Phi\circ\mathcal G_1(v_x)$ lies in $\mathcal G_2((\tilde E_2)_{z})$. We claim that for any $w_x\in(\tilde E_1)_x$ its image $\Phi\circ\mathcal G_1(w_x)$ lies in the same subspace $\mathcal G_2((\tilde E_2)_{z})$.

Suppose the contrary, that is there exists a non-zero vector $w_x\in(\tilde E_1)_x$ such that its image $\Phi\circ\mathcal G_1(w_x)$ lies in $\mathcal G_2((\tilde E_2)_{y})$, where $z\ne y$. Then, we see that 
$$
\Phi\circ\mathcal G_1(w_x-v_x)\in\mathcal G_2((\tilde E_2)_{z})\oplus\mathcal G_2((\tilde E_2)_{y}).
$$
Since the vectors $v_x$ and $w_x$ are different, arguing as above, we may find another point $q\in\tilde M_2$ such that $\Phi\circ\mathcal G_1(w_x-v_x)$ lies in $\mathcal G_2((\tilde E_2)_{q})$. Now by Lemma~\ref{sum} we conclude that the point $q$ coincides with either $z$ or $y$, and in each case it is straightforward to arrive at a contradiction. For example, if $q=z$, we immediately conclude that the vector
$$
\Phi\circ\mathcal G_1(w_x)=\Phi\circ\mathcal G_1(w_x-v_x)+\Phi\circ\mathcal G_1(v_x)
$$
lies in the image $\mathcal G_2((\tilde E_2)_{z})$, and by Lemma~\ref{imm}, the points $z$ and $y$ coincide. 
\end{proof}

The last lemma shows that the operator
$$
J_x=\mathcal G_2^{-1}\circ\Phi\circ\mathcal G_1: (\tilde E_1)_x\longrightarrow (\tilde E_2)_{z}
$$
is defined for all $x$ in a neighbourhood $O_1$ of the point $x_1$, and by the discussion in Step~2, has to coincide with the map $H_2\circ H_1^{-1}$ in a neighbourhood $H_1(\mathcal O)\cap\tilde\pi_1^{-1}(O_1)$ of a given point $v_{x_1}\in (\tilde E_1)_{x_1}$. Without loss of generality we may assume that the last set is connected. Since the map $H_2\circ H_1^{-1}$ is real-analytic, and is an isometry on the open subset 
$$
\Omega_1\cap H_1(\mathcal O)\cap\tilde\pi_1^{-1}(O_1), 
$$
where $\Omega_1$ is given by~\eqref{omega1}, by unique continuation we conclude that it is an isometry on $H_1(\mathcal O)\cap\tilde\pi_1^{-1}(O_1)$. Performing this argument for all points $v_{x_1}$ from the unit sphere in the fibre $(\tilde E_1)_{x_1}$, we conclude that the operator $J_x$ is an isometry for all $x$ in a neighbourhood of $x_1$. This immediately yields a contradiction with the assumption $B_1\ne\tilde M_1$, since the point $x_1\in\tilde M_1$ has been chosen on the boundary $\partial B_1$. Thus, we conclude that the set $B_1$ coincides with the whole manifold $\tilde M_1$.

\medskip
\noindent\underline{\em Step~4.} Now we collect final conclusions. First, relation~\eqref{J:def} defines the fibre preserving map $J:\tilde E_1\to\tilde E_2$. By the argument in Step~2 we see that locally it can be written in the form $H_2\circ H_1^{-1}$, and hence, is smooth, and by Lemma~\ref{step2} is real-analytic. Since by definition it is an isomorphism on each fibre, we conclude that it is a real-analytic vector bundle isomorphism. In particular, it covers a real analytic map $j:\tilde M_1\to\tilde M_2$. Since it is an isometry on each fibre, it is a Euclidean vector bundle isomorphism.

Note that the isomorphism $J$ coincides with the isomorphism $\Phi$ on fibres over $W_1\subset\tilde M_1$. Since the latter is a gauge equivalence, the connections $J^*\tilde\nabla^2$ and $\tilde\nabla^1$ coincide on $W_1$, and since they are real-analytic and $\tilde M_1$ is connected, they coincide on $\tilde M_1$. Similarly, the map $j:\tilde M_1\to\tilde M_2$ coincides with the isometry $\Psi$ on $W_1$, that is the real-analytic metrics $j^*\tilde g_2$ and $\tilde g_1$ coincide on $W_1$, and hence, they coincide everywhere on $\tilde M_1$. Thus, the vector bundle isomorphism $J$ is indeed a gauge equivalence that covers an isometry.
\qed

\section{Vector bundles over surfaces}
\label{2dim}
\subsection{Proof of Theorem~\ref{t2}}
First, note that we may choose a real-analytic metric $g\in c$. Indeed, viewing $M$ as a domain in a closed surface, by the uniformisation theorem, it is straightforward to see  that there is a constant Gauss curvature  metric $g\in c$. In isothermal coordinates such a metric is real-analytic, and hence, is real-analytic with respect to the real-analytic atlas determined by $c$. Throughout the rest of the section we always use this metric on $M$; first, to define the Dirichlet-to-Neumann operators, and second, to describe the immersions by Green kernels. The construction below follows closely the lines in Section~\ref{immersions}, but uses Proposition~\ref{loc2} instead of Proposition~\ref{loc}.

Fix a point $p\in\Sigma\subset\partial M$. We may view the surface $M$ as a subset of the larger surface $\tilde M$,  obtained by gluing the Euclidean disc $D(0,\rho)$ to $M$ such that the points in the half-disc are identified with the points in $M$ by means of boundary normal coordinates centred at  $p$. Since these coordinates, constructed using our metric $g$, lie in the fixed real-analytic atlas, we conclude that this construction yields the real-analytic structure on $\tilde M$, extending the one on $M$. Making $\rho$ smaller, if necessary, we may assume that the metric $g$ extends to a real-analytic metric $\tilde g$ on $\tilde M$. Below by $W$ we denote the neighbourhood of $p$ in $\tilde M$ that corresponds to the disc $D(0,\rho)$.

Let $E_1$ and $E_2$ be two real-analytic vector bundles over $M$, and suppose that for a given open set $\Sigma\subset\partial M$ there exists a vector bundle isomorphism 
$$
\phi:\left. E_1\right|_{\Sigma}\to \left. E_2\right|_{\Sigma}
$$
that covers the identity map of $\Sigma$ and intertwines with the Dirichlet-to-Neumann operators $\Lambda_{1,\Sigma}$ and $\Lambda_{2,\Sigma}$. Making $W$ smaller, if necessary, we can choose orthonormal frames defined over $W\cap\Sigma$ that are related by $\phi$. They identify trivialisations of $\left. E_1\right|_{\Sigma}$ and $\left. E_2\right|_{\Sigma}$, and extending the latter as trivial vector bundles over $W$, we obtain vector bundles $\tilde E_1$ and $\tilde E_2$ over $\tilde M$. Further, identifying boundary normal frames of $\tilde E_1$ and $\tilde E_2$, we extend $\phi$ to a Euclidean isomorphism
\begin{equation}
\label{phi:ext}
\Phi:\left. E_1\right|_{W}\to \left. E_2\right|_{W}.
\end{equation}
By Proposition~\ref{loc2} the real-analytic connection matrices of $\nabla^1$ and $\nabla^2$ coincide in such frames, and we conclude that $\Phi$ is a gauge equivalence.

As in the proof of Theorem~\ref{t1}, by $U$ we denote the open set $\tilde M\backslash\bar M$, and by $\mathcal E$ the trivial vector bundle over $U$, which coincides with both restrictions of $\tilde E_1$ and $\tilde E_2$ to $U$. We use the same notation $\tilde E_i^0$ for the vector bundle $\tilde E_i$ with the removed zero section over the interior of $\tilde M$. Let $\tilde G_i$ be the Dirichlet Green kernel for the connection Laplacian on $\tilde E_i$ with respect to the metric $\tilde g$ on $\tilde M$. For each $i=1,2$ we define the map $\mathcal G_i:\tilde E_i\to W^{-1,2}(\mathcal E)$ by setting
\begin{equation}
\label{g2dim:def}
\tilde E_x\ni v_x\longmapsto \langle v_x,\tilde G_i(x,\cdot)\rangle_{x,\tilde E}\in W^{-1,2}(\mathcal E),
\end{equation}
where $x\in\tilde M$. It satisfies the conclusions in Lemmas~\ref{smooth}-\ref{sum}.

The repetition of the argument in Section~\ref{immersions} shows that Theorem~\ref{t2} is a consequence of the following statement.
\begin{theorem}
\label{t2a}
Under the hypotheses of Theorem~\ref{t2}, consider the maps $\mathcal {G}_i:\tilde E_i\to W^{-1,2}(\mathcal E)$ defined by~\eqref{g2dim:def}, where $i=1,2$. Suppose that the vector bundle isomorphism $\Phi:\mathcal E\to\mathcal E$, described above, intertwines with the $\mathcal G_i$'s, that is 
$$
%\label{intert2}
\mathcal G_2\circ\Phi=\Phi\circ\mathcal G_1\qquad\text{ on~ }~\mathcal E. 
$$
Then the images $\mathcal G_2(\tilde E_2^0)$ and $\Phi\circ\mathcal G_1(\tilde E_1^0)$ coincide as subsets in $W^{-1,2}(\mathcal E)$, and the map $\mathcal G_2^{-1}\circ\Phi\circ\mathcal G_1:\tilde E_1^0\to \tilde E_2^0$ extends  to a real-analytic vector bundle isomorphism $J:\tilde E_1\to\tilde E_2$ that covers the identity map of $M$ and such that $J^*\tilde\nabla^2=\tilde\nabla^1$.
\end{theorem}

Assuming that the construction described above is the partial case of the setup in Section~\ref{immersions}, such that the extended isomorphism $\Phi$ in~\eqref{phi:ext} covers the identity map, Theorem~\ref{t2a} can actually be derived from Theorem~\ref{t1a}. More precisely, the only statement that needs to be checked is that the isomorphism $J$ in Theorem~\ref{t1a} covers the identity map of $M$. The latter is a consequence of the unique continuation property of real-analytic maps, since $J$ coincides with $\Phi$ over $W$ and the latter covers the identity map.

However, under the hypotheses of Theorem~\ref{t2}, and consequently, of Theorem~\ref{t2a}, the original argument in Section~\ref{immersions} can be simplified, and for reader's convenience we prefer to outline this below, highlighting the differences. 

\subsection{Outline of the proof of Theorem~\ref{t2a}}
Let $B\subset\tilde M$ be the largest connected open set containing the fixed point $p\in\Sigma$ and such that for any $x\in B$ the images of the fibres $\Phi\circ\mathcal G_1((\tilde E_1)_x)$ and $\mathcal G_2((\tilde E_2)_{x})$ coincide. Then, by Lemma~\ref{imm} the composition
\begin{equation}
\label{2:J:def}
J_x=\mathcal G_2^{-1}\circ\Phi\circ\mathcal G_1: (\tilde E_1)_x\longrightarrow (\tilde E_2)_{x}
\end{equation}
is an isomorphism of the fibres, and defines a fibre preserving map $J:\left.\tilde E_1\right|_{B}\to\tilde E_2$. Repeating the argument in Section~\ref{immersions}, we see that $W\subset B$. The aim is to show that $B$ coincides with the interior of $\tilde M$. Note that unlike in the proof of Theorem~\ref{t1a} we do not ask $J_x$ to be an isometry on the fibres. This is related to the fact that Step~1 of the proof is trivial in this case, since we are not trying to determine the topology of $M$.

Suppose the contrary, $B\ne\tilde M$, and pick a point $\bar x\in\partial B$ that lies in the interior of $M$. By continuity of $\Phi\circ\mathcal G_1$ and $\mathcal G_2$, together with Lemma~\ref{imm}, we conclude that the map $J$ extends to $\bar x$. Then there exist non-zero vectors $v_{\bar x}\in(\tilde E_1)_{\bar x}$ and $w_{\bar x}\in(\tilde E_2)_{\bar x}$ such that the images $\Phi\circ\mathcal G_1(v_{\bar x})$ and $\mathcal G_2(w_{\bar x})$ coincide. Denoting this vector by $u$, we can repeat the argument in the proof of Theorem~\ref{t1a}, see Step~2 and Lemma~\ref{step3a} in Step~3, to conclude that the images of $\Phi\circ\mathcal G_1$ and $\mathcal G_2$ coincide around $u$. Then, we have the following version of Lemma~\ref{step3b}.
\begin{lemma}
There is a connected neighbourhood $O$ of the point $\bar x\in\tilde M$ such that for any $x\in O$  the images of fibres $\Phi\circ\mathcal G_1((\tilde E_1)_x)$ and $\mathcal G_2((\tilde E_2)_{x})$ coincide.
\end{lemma}
\begin{proof}
Following the notation and the line of argument in Section~\ref{immersions}, we choose a neighbourhood $O\subset \tilde\pi_1\circ H_1(\mathcal O)$ of $\bar x\in\tilde M$. For a given point $x\in O$ and a vector $v_x\in H_1(\mathcal O)$ denote by $w_z$ the vector $H_2\circ H_1^{-1}(v_x)$. It satisfies the relation
$$
\Phi\circ\mathcal G_1(v_x)=\mathcal G_2(w_z),
$$
and we claim that the point $z$ has to coincide with $x$. Indeed, this follows from the fact that the map $H_2\circ H_1^{-1}$ coincides with $J$ on the open set $\mathcal O\cap\Omega_1$, where $\Omega_1$ is defined by relation~\eqref{omega1}, and in our setup $J$ covers the identity map on the base, that is
\begin{equation}
\label{uc:proj}
\tilde\pi_2\circ H_2\circ H_1^{-1}=\tilde\pi_1\qquad\text{on~ }\quad\mathcal O\cap\Omega_1.
\end{equation}
By Lemma~\ref{step2} the map $H_2\circ H_1^{-1}$ is real-analytic, and so are the vector bundle projections $\tilde\pi_1$ and $\tilde\pi_2$. Thus, by unique continuation we conclude that relation~\eqref{uc:proj} holds on $\mathcal O$, and hence,
$$
z=\tilde\pi_2(w_z)=\tilde\pi_1(v_x)=x.
$$
Now by linearity, it is straightforward to see that the image  of the fibre $\Phi\circ\mathcal G_1((\tilde E_1)_x)$ lies in $\mathcal G_2((\tilde E_2)_{x})$. Since these images are vector subspaces of the same dimension, we are done.
\end{proof}
The last lemma yields a contradiction, and we conclude that $B=\tilde M$. Thus, relation~\eqref{2:J:def} defines the fibre preserving map $J:\tilde E_1\to\tilde E_2$, which covers the identity map of $M$. Other remaining properties of $J$ are obtained by repeating the argument at the end of Section~\ref{immersions}.

\section{Proofs of auxiliary results}
\label{auxs}
\subsection{Proof of Lemma~\ref{smooth}}
Since the map $\mathcal G:\tilde E\to W^{\ell,2}(\mathcal E)$ is linear on each fibre, for a proof of the lemma it is sufficient to show that the map that sends a point $x\in\tilde M$ to the function $\tilde G(x,\cdot)$, viewed as an element in the Sobolev space $W^{\ell,2}$, is smooth. Below we assume that $x$ ranges in a chart on $\tilde M$ where the vector bundle $\tilde E$ is trivial. First, we claim that for any section $\varphi\in W^{-\ell,2}_0(\mathcal E)$ the section
$$
\tilde\psi(x)=\int_{\tilde M}\langle\tilde G(x,y),\tilde\varphi(y)\rangle_yd\vol(y),
$$
is differentiable, where $\tilde\varphi$ is an extension of $\varphi$ by zero, and for any $h\in\mathbb R^n$ the linear functional 
\begin{equation}
\label{diff:def}
\varphi\longmapsto\sum_ih_i\frac{\partial}{\partial x^i}\tilde\psi(x)=\sum_ih_i\frac{\partial}{\partial x^i}\int_{\tilde M}\langle\tilde G(x,y),\tilde\varphi(y)\rangle_yd\vol(y)
\end{equation}
defines an element in $W^{\ell,2}(\mathcal E)$. Indeed, by standard theory the section $\tilde\psi$ can be viewed as the solution to the Dirichlet problem
$$
\Delta^{\tilde E}\tilde\psi=\tilde\varphi,\qquad\left.\tilde\psi\right|_{\partial\tilde M}=0,
$$
in the Sobolev space $W^{-\ell+2,2}_0$, and since $\ell<1-n/2$, it lies in the Holder space $C^{1,\alpha}$ for some $\alpha>0$. To show that the functional defined by~\eqref{diff:def} lies in $W^{\ell,2}(\mathcal E)$, it is sufficient to show that
$$
\abs{\sum_ih_i\frac{\partial}{\partial x^i}\tilde\psi(x)}\leqslant C\abs{h}\abs{\varphi}_{-\ell,2}
$$
for some constant $C>0$, where $\abs{~\cdot~}_{-\ell,2}$ stands for the Sobolev norm. The latter is a direct consequence of the inequality $\abs{\tilde\psi}_{-\ell+2,2}\leqslant C'\abs{\tilde\varphi}_{-\ell,2}$, which follows from standard theory, together with the Sobolev embedding theorem. Thus, we obtain the linear operator
$$
L_x:\mathbb R^n\ni h\longmapsto \sum_ih_i\frac{\partial}{\partial x^i}\tilde G(x,\cdot)\in W^{\ell,2}(\mathcal E),
$$
and claim that it is the differential of the map $x\mapsto\tilde G(x,\cdot)$. In other words, we claim that for any $\varepsilon>0$ the inequality 
$$
\abs{\tilde G(x+h,\cdot)-\tilde G(x,\cdot)-L_x(h)}_{\ell,2}\leqslant\varepsilon\abs{h}
$$
holds, for any $h\in\mathbb R^n$ such that $\abs{h}<\delta$ for an appropriate $\delta>0$. In the notation above, for the latter it is sufficient to show that
\begin{equation}
\label{diff:in}
\abs{\tilde\psi(x+h)-\tilde\psi(x)-\sum_ih_i\frac{\partial}{\partial x^i}\tilde\psi(x)}\leqslant C\abs{h}^{1+\alpha}\abs{\varphi}_{-\ell,2}
\end{equation}
for some positive constants $C$ and $\alpha$, and arbitrary $h\in\mathbb R^n$. Recall the so-called Hadamard formula for $C^1$-smooth functions:
$$
\tilde\psi(x+h)-\tilde\psi(x)=\sum_i\gamma_i(x)h_i,\quad\text{where}\quad\gamma_i(x)=\int\limits_0^1\frac{\partial\tilde\psi}{\partial x^i}(x+th)dt,
$$
and we use a trivialisation of $\tilde E$ to view sections around $x$ as vector functions. Using this relation, we obtain
\begin{multline*}
\abs{\tilde\psi(x+h)-\tilde\psi(x)-\sum_ih_i\frac{\partial}{\partial x^i}\tilde\psi(x)}=\abs{\sum_i(\gamma_i(x)-\frac{\partial}{\partial x^i}\tilde\psi(x))h_i}\\
\leqslant\abs{h}\left(\int_0^1\abs{D\tilde\psi(x+th)-D\tilde\psi(x))}^2dt\right)^{1/2}\!\!\leqslant\abs{\tilde\psi}_{C^{1,\alpha}}\abs{h}^{1+\alpha}\leqslant C''\abs{\tilde\psi}_{-\ell+2,2}\abs{h}^{1+\alpha}\\
\leqslant C'C''\abs{\tilde\varphi}_{-\ell,2}\abs{h}^{1+\alpha}\leqslant C\abs{\varphi}_{-\ell,2}\abs{h}^{1+\alpha},
\end{multline*}
where in the second inequality we estimate the integral via the Holder norm and $\abs{h}^\alpha$, and in the third we use the Sobolev embedding theorem. Thus, relation~\eqref{diff:in} is demonstrated, and we conclude that the map $x\mapsto\tilde G(x,\cdot)$ is differentiable. Finally, for a proof that it is smooth, it remains to show that the map $x\mapsto L_x$ is continuous. The latter is a consequence of the inequality
$$
\abs{\sum_i(\frac{\partial}{\partial x^i}\tilde\psi(x_1)-\frac{\partial}{\partial x^i}\tilde\psi(x_2))h_i}\leqslant C\abs{h}\abs{x_1-x_2}^\alpha\abs{\varphi}_{-\ell,2}
$$
for some positive constants $C$ and $\alpha$, which can be proved in a fashion similar to the one above. Thus, we are done.
\qed

\subsection{Proof of Lemma~\ref{imm}}
First, we show that the map $\mathcal G$ is a linear embedding on each fibre $\tilde E_x$. For otherwise, there exists a point  $x$ in the interior of $\tilde M$ and a non-zero vector $v_x\in \tilde E_x$ such that the product $\langle v_x,\tilde G(x,\cdot)\rangle_x$ equals zero in $W^{\ell,2}(\mathcal E)$. The latter in particular implies that
\begin{equation}
\label{imm:cont}
\langle v_x,\tilde G(x,y)\rangle_x=0\qquad\text{ for all }\quad y\in U\backslash\{x\}.
\end{equation}
Since the left-hand side above is real-analytic, we conclude that relation~\eqref{imm:cont} continues to hold on $\tilde M\backslash\{x\}$. Now let $s\in\mathcal D(\tilde E)$ be a compactly supported section such that $s(x)=v_x$. Then, we obtain
$$
0=\int_{\tilde M}\langle\langle v_x,\tilde G(x,y)\rangle_x,\Delta^{\tilde E} s(y)\rangle_yd\vol(y)=\langle v_x,\int_{\tilde M}\langle\tilde G(x,y),\Delta^{\tilde E}s(y)\rangle_yd\vol(y)\rangle_x=\langle v_x,v_x\rangle,
$$
where we changed the order of operations in independent variables $x$ and $y$ in the second relation, and used the definition of the Dirichlet Green kernel in the third. Thus, we conclude that the vector $v_x$ has to vanish, and the kernel of a linear operator given by~\eqref{g:def} is trivial, that is the map $\mathcal G$ is indeed a linear embedding on each fibre.

A similar argument shows that the map $\mathcal G$ is injective everywhere on $\tilde E^0$. Indeed, suppose that there exist two points $x_1$ and $x_2$ in the interior of $\tilde M$ and non-zero vectors $v_{x_1}$ and $v_{x_2}$ in the fibres over them such that $\langle v_{x_1},\tilde G(x_1,\cdot)\rangle$ and $\langle v_{x_2},\tilde G(x_2,\cdot)\rangle$ coincide in $W^{\ell,2}(\mathcal E)$. Then, it is straightforward to see that
\begin{equation}
\label{imm:h1}
\langle v_{x_1},\tilde G(x_1,y)\rangle_x=\langle v_{x_2},\tilde G(x_2,y)\rangle_x\qquad\text{ for all }\quad y\in U\backslash\{x_1,x_2\}.
\end{equation}
As above, by unique continuation we may assume that relation~\eqref{imm:h1} holds for all $y\in\tilde M\backslash\{x_1,x_2\}$. In addition, since the map $\mathcal G$ is injective on fibres, we may assume that $x_1\ne x_2$. Let $s\in\mathcal D(\tilde E)$ be a compactly supported section such that $s(x_1)=v_{x_1}$ and $s(x_2)=0$. Then, we obtain
\begin{multline*}
\langle v_{x_1},v_{x_1}\rangle=\langle v_{x_1},\int_{\tilde M}\langle\tilde G(x_1,y),\Delta^{\tilde E}s(y)\rangle_yd\vol(y)\rangle_x=\int_{\tilde M}\langle\langle v_{x_1},\tilde G(x_1,y)\rangle_x,\Delta^{\tilde E}s(y)\rangle_yd\vol(y)\\=
\int_{\tilde M}\langle\langle v_{x_2},\tilde G(x_2,y)\rangle_x,\Delta^{\tilde E}s(y)\rangle_yd\vol(y)=\langle v_{x_2},\int_{\tilde M}\langle\tilde G(x_2,y),\Delta^{\tilde E}s(y)\rangle_yd\vol(y)\rangle_x\\=
\langle v_{x_2},s(x_2)\rangle=\langle v_{x_2},0\rangle=0.
\end{multline*}
Thus, the vector $v_{x_1}$ vanishes, and we arrive at a contradiction.

Finally, to show that the map $\mathcal G$ is an immersion we analyse its differential $D_{v_x}:T_{v_x}\tilde E\to W^{\ell,2}(\mathcal E)$. First, note that a connection on the vector bundle $\tilde E$ defines the decomposition of the tangent space $T_{v_x}\tilde E$ as the direct sum $H_{v_x}\oplus\tilde E_x$, where $H_{v_x}$ is the so-called horizontal subspace, see~\cite{GKM}. Since the differential of the projection $\tilde\pi:\tilde E\to\tilde M$ establishes an isomorphism $D_{v_x}\tilde\pi:H_{v_x}\to T_x\tilde M$, we may view tangent vectors from $T_{v_x}\tilde E$ as pairs $(X,\xi)$, where $X\in T_xM$, and $\xi\in\tilde E_x$. With these identifications, it is straightforward to show that
\begin{equation}
\label{imm:d}
D_{v_x}\mathcal G(X,\xi)=\langle v_x,\nabla_X\tilde G(x,\cdot)\rangle_{x,\tilde E}+\langle\xi,\tilde G(x,\cdot)\rangle_{x,\tilde E},
\end{equation} 
where by the covariant derivative $\nabla_X\tilde G(x,\cdot)$ we mean the derivative with respect to the variable $x$ on $\tilde E\boxtimes\tilde E$, that is given by $\nabla_X(u_x\otimes u_y)=\nabla^{\tilde E}_Xu_x\otimes u_y$. Now choosing appropriate test-sections in the fashion similar to the one above, it is straightforward to show that the differential $D_{v_x}\mathcal G$ is injective.  In more detail, assume that the right-hand side of  relation~\eqref{imm:d} equals zero for some $X\in T_xM$ and $\xi\in\tilde E_x$. Then, by unique continuation we may assume that
\begin{equation}
\label{imm:step}
\langle v_x,\nabla_X\tilde G(x,y)\rangle_{x,\tilde E}+\langle\xi,\tilde G(x,y)\rangle_{x,\tilde E}=0\qquad\text{for all}\quad y\in\tilde M\backslash\{x\}.
\end{equation}
Now choosing a compactly supported section $s\in\mathcal D(\tilde E)$ such that $s(x)=\xi$ and $\left.\nabla_X^{\tilde E}s\right|_x=0$, we obtain
\begin{multline*}
0=\int_{\tilde M}\langle\langle v_x,\nabla_X\tilde G(x,y)\rangle_x,\Delta^{\tilde E}s(y)\rangle_yd\vol(y)+\int_{\tilde M}\langle\langle\xi,\tilde G(x,y)\rangle_x,\Delta^{\tilde E}s(y)\rangle_yd\vol(y)\\
=\langle v_x,\nabla_X\int_{\tilde M}\langle\tilde G(x,y),\Delta^{\tilde E}s(y)\rangle_yd\vol(y)\rangle_x+\langle \xi,\int_{\tilde M}\langle\tilde G(x,y),\Delta^{\tilde E}s(y)\rangle_yd\vol(y)\rangle_x\\
=\langle v_x,\nabla_X^{\tilde E}s\rangle+\langle\xi,\xi\rangle=0+\langle\xi,\xi\rangle.
\end{multline*}
Thus, the vector $\xi\in\tilde E_x$ vanishes, and by relation~\eqref{imm:step} we conclude that the term $\langle v_x,\nabla_X\tilde G(x,\cdot)\rangle_x$ equals zero. Now choosing a test-section $s\in\mathcal D(\tilde E)$ such that $\left.\nabla_Xs\right|_x=v_x$, it is straightforward to see that the vector $v_x$ equals zero as well. Thus, the differential $D_{v_x}\mathcal G$ is indeed injective, and we are done.
\qed

\subsection{Proof of Lemma~\ref{sum}}
Suppose the contrary, the point $x$ does not coincide neither with $q_1$ nor with $q_2$. Then there exist non-zero vectors $v_x\in\tilde E_x$, $w_{q_1}\in\tilde E_{q_1}$, and $w_{q_2}\in\tilde E_{q_2}$ such that
\begin{equation}
\label{sum:aux}
\langle v_x,\tilde G(x,y)\rangle_{x,\tilde E}=\langle w_{q_1},\tilde G(q_1,y)\rangle_{q_1,\tilde E}+\langle w_{q_2},\tilde G(q_2,y)\rangle_{q_2,\tilde E}
\end{equation}
for any $y\in U$. Since both parts of this identity are real-analytic functions of $y$, we conclude that it continues to hold for all $y$ in the complement of the points $x$, $q_1$, and $q_2$ in $\tilde M$. Since $x$ does not coincide neither with $q_1$ nor with $q_2$, there exists a smooth section $s\in\mathcal D(\tilde E)$ whose support does not contain $q_1$ and $q_2$, and such that $s(x)=v_x$. Then, by definition of the Dirichlet Green kernel we obtain
\begin{multline*}
\langle v_x,v_x\rangle=\langle v_x,\int_{\tilde M}\langle\tilde G(x,y),\Delta^{\tilde E} s(y)\rangle_yd\vol(y)\rangle_x=\int_{\tilde M}\langle\langle v_x,\tilde G(x,y)\rangle_x,\Delta^{\tilde E}s(y)\rangle_yd\vol(y)\\=
\int_{\tilde M}\langle\langle w_{q_1},\tilde G(q_1,y)\rangle_x,\Delta^{\tilde E}s(y)\rangle_yd\vol(y)+\int_{\tilde M}\langle\langle w_{q_2},\tilde G(q_2,y)\rangle_x,\Delta^{\tilde E}s(y)\rangle_yd\vol(y)\\=
\langle w_{q_1},\int_{\tilde M}\langle\tilde G(q_1,y),\Delta^{\tilde E}s(y)\rangle_yd\vol(y)\rangle_x+\int_{\tilde M}\langle w_{q_2},\langle\tilde G(q_2,y),\Delta^{\tilde E}s(y)\rangle_yd\vol(y)\rangle_x\\=
\langle w_{q_1},s(q_1)\rangle+\langle w_{q_2},s(q_2)\rangle=0,
\end{multline*}
where we used the relation $s(x)=v_x$ in the first equality, identity~\eqref{sum:aux} in the third, and the fact that the section $s$ is chosen so that it vanishes at $q_1$ and $q_2$ in the last. Thus, we conclude that $v_x$ equals zero and arrive at a contradiction.
\qed

\subsection*{Acknowledgements}  
Our interest in inverse problems has been provoked by Slava Kurylev in a few brief discussions, spread over a number of years. The second-named author is grateful to Lauri Oksanen for discussions on related topics and useful references. Finally, we would like to thank the referee for pointing out some inaccuracies and additional references.

\end{document}